\documentclass[11pt,oneside]{amsart}
\usepackage{amssymb}

\usepackage{amsfonts,amsmath,amstext,amsbsy,euscript,amsthm}
\usepackage{mathtools}
\usepackage{enumerate}
\usepackage{mathabx}
\usepackage[all]{xy}
\usepackage[utf8]{inputenc} 
\usepackage{marginnote}
\usepackage{color}
\usepackage{tikz-cd}
\usepackage{mathrsfs}
\usepackage[colorlinks]{hyperref}

\theoremstyle{plain}
\newtheorem{theorem}{Theorem}[section]
\newtheorem{proposition}[theorem]{Proposition}
\newtheorem{corollary}[theorem]{Corollary}
\newtheorem{lemma}[theorem]{Lemma}

\theoremstyle{definition}

\theoremstyle{remark}
\newtheorem{remark}[theorem]{Remark}
\newcommand{\note}[2][\null]{%
  \marginpar{\renewcommand{\baselinestretch}{1}\vspace{-1em}\hrule\vspace{3pt}%
  \scriptsize\raggedright\textsf{#2\ifx#1\null\else\\\hfill--- 
  {\em #1}\fi}\vspace{1.5em}}%
}

\numberwithin{equation}{section}
\begin{document}

\title[Convolution Operators On Symmetric Spaces]{Surjectivity of Convolution Operators on Noncompact Symmetric Spaces}

\author{Fulton Gonzalez}
\address{Department of Mathematics,
Tufts University,
Medford, MA 02155}
\email{fulton.gonzalez@tufts.edu}
\author{Jue Wang}
\address{Department of Mathematics,
Tufts University,
Medford, MA 02155}
\email{jue.wang@tufts.edu}
\author{Tomoyuki Kakehi}
\address{Institute of Mathematics,
University of Tsukuba,
Tsukuba, Ibaraki, 305-8571, Japan}
\email{kakehi@math.tsukuba.ac.jp}


%

\begin{abstract}
Let $\mu$ be a $K$-invariant compactly supported distribution on a noncompact Riemannian symmetric space $X=G/K$. If the spherical Fourier transform $\widetilde\mu(\lambda)$ is slowly decreasing, it is known that the right convolution operator $c_\mu\colon f\mapsto f*\mu$ maps $\mathcal E(X)$ onto $\mathcal E(X)$. In this paper, we prove the converse of this result. We also prove that $c_\mu$ has a fundamental solution if and only if $\widetilde\mu(\lambda)$ is slowly decreasing.
\end{abstract}

\maketitle

\def\rar{\rightarrow} 
\def\fk{\mathfrak}
\def\e{\mathfrak e}
\def\g{\mathfrak g} 
\def\k{\mathfrak k} 
\def\a{\mathfrak a} 
\def\m{\mathfrak m} 
\def\n{\mathfrak n}
\def\p{\mathfrak p} 
\def\u{\mathfrak u} 
\def\t{\mathfrak t} 
\def\h{\mathfrak h} 
\def\z{\mathfrak z}
\def\d{\mathfrak d}
\def\s{\mathfrak s}
\def\Exp{\text{Exp}}
\def\dbar{\overline\delta}
\def\Ad{\text{Ad}}
\def\ad{\text{ad}}
\def\rr{\mathbb R}
\def\rn{\mathbb R^n}
\def\cn{\mathbb C^n}
\def\nm{\nonumber}
\def\cc{\mathbb C}
\def\zz{\mathbb Z}
\def\67{67}
\def\be{\begin{equation}}
\def\ee{\end{equation}}

\def\Bbb{\mathbb}
\def\Cal{\mathcal}

\section{Slowly Decreasing Functions}

The notion of a slowly decreasing function was introduced by L. Ehrenpreis  in 1960 in connection with the following problem.

Let $\mu$ be a fixed distribution in $\mathcal E'(\rr^n)$. Under what conditions on $\mu$ is the convolution operator
$$
c_\mu\colon f\mapsto f*\mu
$$
surjective as a map from $\mathcal E(\rr^n)$ to $\mathcal E(\rr^n)$, or from $\mathcal D'(\rr^n)$ to $\mathcal D'(\rr^n)$? 

This problem was also studied by B. Malgrange \cite{Malgrange1955} for these and other function and distribution spaces. The aim was to generalize related results obtained by the two authors  when $\mu=D\,\delta_0$, or more generally when $\mu=\sum_j D_j\,\delta_{x_j}$, where $D$ and the $D_j$ are constant coefficient differential operators, and $\{x_j\}$ is a finite set of points in $\rr^n$. (Here $\delta_x$ is the delta distribution at $x$.)

By definition, a function $F$ on $\cc^n$ is \emph{slowly decreasing} provided that there exists a constant $A>0$ such that for any $\xi\in\rr^n$, the open ball in $\cc^n$ centered at $\xi$, with radius $A\,\log (2+\|\xi\|)$ contains a point $\zeta$ for which
\begin{equation}\label{E:slow-decrease-cond}
|F(\zeta)|> (A+\|\xi\|)^{-A}.
\end{equation}

Let us recall that the Paley-Wiener Theorem for distributions in $\rr^n$ states that the Fourier transform $\mu\mapsto \mu^*$ is a linear bijection from the vector space of distributions supported in the closed ball $\overline B_R$ of radius $R$ centered at $0\in\rr^n$ onto the vector space of entire functions on $\cc^n$ of exponential type $R$ and which are slowly (i.e., polynomially) increasing on $\rr^n$. These are precisely the entire functions $F$ on $\cc^n$ for which there is an integer $N\in\zz^+$ and a constant $C$ such that
\begin{equation}\label{E:PW-estimate}
|F(\zeta)|\leq C\,(1+\|\zeta\|)^N\,e^{R\,\|\text{Im}\,\zeta\|},\qquad \zeta\in\cc^n.
\end{equation}
Suppose that $F$ is entire in $\cc^n$ and satisfies the Paley-Wiener estimate \eqref{E:PW-estimate}. If $F$ happens to be slowly decreasing, then in fact the point $\zeta$ satisfying \eqref{E:slow-decrease-cond} can be found in $\rr^n$. (The constant $A$ may need to be increased, but this does not change the condition of slow decrease.) This is a consequence of Theorem 5 in \cite{Ehrenpreis1955}, which Ehrenpreis calls a ``minimum modulus theorem.'' 

In the fourth paper in his series on ``Solutions of Some Problems of Division," published in 1960, Ehrenpreis proved the following result.

\begin{theorem}\label{T:Ehrenpreis60} (\cite{Ehrenpreis1960}, Section 2.) Let $\mu\in\mathcal E'(\rr^n)$. Then the following are equivalent:
\begin{enumerate}[(a)]
\item The Fourier transform $\mu^*(\zeta)$ is slowly decreasing.
\item The convolution operator $c_\mu$ maps $\mathcal E(\rr^n)$ onto $\mathcal E(\rr^n)$.
\item The convolution operator $c_\mu$ maps $\mathcal D'(\rr^n)$ onto $\mathcal D'(\rr^n)$.
\item There is a distribution $S\in\mathcal D'(\rr^n)$ such that $S*\mu=\delta_0$.
\item\label{I:invertible-cond} The linear map $f*\mu\mapsto f$ is continuous from $c_\mu(\mathcal D(\rr^n))$ to $\mathcal D(\rr^n)$.
\end{enumerate}
%
\end{theorem}

Following Ehrenpreis, we will call a distribution $\mu$ \emph{invertible} if it satisfies any of the equivalent conditions in Theorem \ref{T:Ehrenpreis60}.
(The term is appropriate because of the last condition above.) 

We call the distribution $S$ in (d) a \emph{fundamental solution} to $c_\mu$. Theorem \ref{T:Ehrenpreis60} clearly implies the existence of fundamental solutions to constant coefficient differential operators on $\rr^n$, the well-known Malgrange-Ehrenpreis Theorem.

As to condition (c), the key to the proof of the surjectivity of $c_\mu$ for $\mathcal E(\rr^n)$ is the fact that the slow decrease condition is equivalent to the condition that if $T$ is a distribution  in $\mathcal E'(\rr^n)$ such that $T^*/\mu^*$ is entire, then it is slowly increasing on $\rr^n$. Since $T^*/\mu^*$ is necessarily of exponential type (\cite{Malgrange1955}, Theorem 1 or \cite{EhrenpreisMP}, Theorem 4), it must be the Fourier transform of a distribution $S\in\mathcal E'(\rr^n)$. 

There is yet another equivalent condition for $\mu$ to be invertible, due to H\"ormander: the ``propagation'' of singular supports. This is given in Lemma \ref{T:slowdecrease1}, and is crucial to the proof of one of main theorems, Theorem \ref{T:W-surjectivity}.

It is natural to try to consider the analogue of Theorem \ref{T:Ehrenpreis60} to Riemannian symmetric spaces of the noncompact type, at least in the $K$-invariant case. This is the aim of the present paper. As a corollary to one of our main results, we will obtain the existence of fundamental solutions to invariant differential operators on such spaces, a fact which was first proved by Helgason in 1964 (\cite{Helgason1964}).

\section{Notation and Preliminaries}
\subsection{Euclidean Fourier Transforms}

The Fourier transform for appropriate functions $F$ on $\rr^n$ will be denoted by $F^*$:
$$
F^*(\xi)=\int_{\rr^n} f(x)\,e^{-i\langle x,\xi\rangle}\,dx, \qquad \xi\in\rr^n.
$$
The Fourier transform of a compactly supported distribution $S$ in $\rr^n$ will likewise be denoted by $S^*$:
$$
S^*(\zeta)=\int_{\rr^n} e^{-i\langle x,\zeta\rangle}\,dS(x),\qquad\zeta\in\cc^n.
$$
We say that an entire function $\Psi$ on $\cc^n$ is \emph{of exponential type} $A\geq 0$ if there is a constant $C$ such that $|\Psi(\zeta)|\leq C\,e^{A\,\|\zeta\|}$ for all $\zeta\in\cc^n$. The \emph{Paley-Wiener Theorem for functions} states that the Fourier transform is a linear bijection from $\mathcal D(\rr^n)$ onto the vector space of entire functions $F$ on $\cc^n$ of exponential type which are rapidly decreasing on $\rr^n$. These are precisely the entire functions $F$ for which there exists an $A\geq 0$ such that
$$
\sup_{\zeta\in\cc^n} |F(\zeta)|\,(1+\|\zeta\|)^N\,e^{-A\,\|\text{Im}\,\zeta\|}<\infty,\qquad N\in\mathbb Z^+.
$$
(See \cite{Ho2}, Theorem 16.3.10 for a further exposition on ``exponential type.'') 
We can topologize the range $\mathcal D(\rr^n)^*$ so that the Fourier transform is a homeomorphism. There are more ``intrinsic'' characterizations of this topology on $\mathcal D(\rr^n)^*$, which are in fact used to prove Theorem \ref{T:Ehrenpreis60}. See, for example, \cite{Ehrenpreis1956}, Theorem 1, or \cite{Helgason2011}, Theorem 4.9.

\subsection{Test Function and Distribution Spaces on Manifolds.}
The topology of $\mathcal E(M)$ and of $D(M)$, for $M$ any manifold, is discussed in detail in \cite{GGA}, Chapter II. To summarize, $\mathcal E(M)$ is the vector space $C^\infty(M)$, endowed with the topology of uniform convergence on compact sets in $M$ of all derivatives $Df$, for $f\in\mathcal C^\infty(M)$, where $D$ is any $C^\infty$ linear differential operator on $M$. Since $M$ is $\sigma$-compact, $\mathcal E(M)$ is easily seen to be a Fr\'echet space.

An alternative description of the topology of $\mathcal E(M)$ is as follows. If $U$ is any coordinate neighborhood in $M$, we can think of $U$ as an open subset of $\mathbb R^n$, and endow $\mathcal E(U)$ with its usual Fr\'echet space topology. The space $\mathcal E(M)$ is then given the weakest topology that makes the restriction maps $f\mapsto f|_U$ continuous, for all coordinate neighborhoods $U$ in $M$.

For any compact subset $B$ of $M$, the space $\mathcal D(B)$ of $C^\infty$ functions supported in $B$ is given the topology inherited from $\mathcal E(M)$. $\mathcal D(B)$ is then a Fr\'echet space.  The space $\mathcal D(M)$ is the vector space $C_c^\infty(M)$, endowed with the inductive limit topology generated by the $\mathcal D(B)$. In particular, a convex set $W$ in $\mathcal D(M)$ is a neighborhood of $0$ if and only if $W\cap \mathcal D(B)$ is a $0$-neighborhood in $\mathcal D(B)$ for all $B$. 

Since $M$ is second countable, it is the union of a nested increasing sequence $\{U_j\}$ of relatively compact open sets, and we may take the Fr\'echet spaces $\mathcal D(\overline U_j)$ to be a sequence of definition of $\mathcal D(M)$. The topology of $D(M)$ is independent of the choice of the sequence $\{U_j\}$. If $M$ is a complete Riemannian manifold, we may take the $U_j$ to be open balls of radius $j$.


In case $M$ is a Riemannian manifold for which the exponential map at a given point $p$ is a diffeomorphism (such as when $M$ is a Riemannian symmetric space of the noncompact type), then the spaces $\mathcal D(M)$ and $\mathcal E(M)$ can be naturally identified with the corresponding Euclidean test function spaces on the tangent space to $M$ at $p$.

The dual spaces $\mathcal D'(M)$ and $\mathcal E'(M)$ are the spaces of distributions and compactly supported distributions on the manifold $M$. We will endow these distribution spaces with their strong dual topologies (of uniform convergence on bounded subsets of $\mathcal D(M)$ and $\mathcal E(M)$, respectively). In the analysis that we will carry out in this paper, we could just as easily use the weak*-topologies on these dual spaces, since for both the strong and the weak*-topologies on $\mathcal D'(M)$ and $\mathcal E'(M)$, the convergent sequences are the same (\cite{TrevesTVS}, 34.4, Corollary 2), the closed convex sets are the same (\cite{TrevesTVS}, Proposition 36.2),  and, since $\mathcal D(M)$ and $\mathcal E(M)$ are reflexive, the bounded sets are the same (\cite{TrevesTVS}, Theorem 36.2).

\subsection{Convolutions on Lie Groups and Homogeneous Spaces.}
Let $G$ be a unimodular Lie group with Haar measure $du$, normalized so $G$ has unit measure in case $G$ is compact. If $\phi$ and $\psi$ are appropriate functions on  $G$ (e.g., both $C^\infty$ and one with compact support), the convolution $\phi*\psi$ is given by
$$
\phi*\psi(g)=\int_G \phi(u)\,\psi(u^{-1}g)\,du=\int_G \phi(gu^{-1})\,\psi(u)\,du,\qquad g\in G.
$$
If $\phi$ is a $C^\infty$ function and $T$ a distribution on $G$, one of them with compact support, then convolutions $\phi*T$ and $T*\phi$ are the $C^\infty$ functions on $G$ given by
\begin{align*}
\phi*T(g)&=\int_G \phi(gu^{-1})\,dT(u),\\
T*\phi(g)&=\int_G \phi(u^{-1}g)\,dT(u),\qquad g\in G.
\end{align*}
In the above we have used the integral convention for distributions.
If $S$ and $T$ are distributions on $G$, one with compact support, then the convolution $S*T$ is the distribution on $G$ given by
$$
S*T(\phi)=\int_{G\times G} \phi(uv)\,dS(u)\,dT(v),\qquad \phi\in\mathcal D(G).
$$
We interpret the above as an iterated integral; the order does not matter.

If $\phi$ is any function on $G$, we let $\widecheck\phi(g)=\phi(g^{-1})$. Likewise, if $S$ is any distribution on $G$, we define the distribution $\widecheck S$ by $\widecheck S(\phi)=S(\widecheck\phi)$ for all $\phi\in\mathcal D(G)$.

For a fixed distribution $S$, the map $T\mapsto S*T$ is the adjoint to the map $\phi\mapsto \widecheck S*\phi$, and for fixed $T$, the map $S\mapsto S*T$ is the adjoint to $\phi\mapsto \phi*\widecheck T$. Here the appropriate test function space that $\phi$ belongs to depends on whether the distribution it is being convolved with has compact support. In any event, the bilinear map $(S,T)\mapsto S*T$ is therefore separately continuous (say, from $\mathcal D'(G)\times\mathcal E'(G)$ to $\mathcal D'(G)$). Unfortunately, this map is \emph{not} jointly continuous. (See \cite{TrevesTVS}, Ch.41.) By the Banach-Steinhaus Theorem, it is, however, hypocontinuous; that is to say, equicontinuous in each factor over fixed bounded subspaces of the other factor.     

Convolutions in $G$ are commutative if and only if $G$ is abelian. 

Now suppose that $K$ is a compact subgroup of $G$.  Convolution of functions or distributions on $G/K$ is carried out by lifting them to $G$. The convolution calculus on $G/K$ is discussed at length in \cite{GGA}, Ch. II, \S5. 

To summarize, let $\pi\colon G\to G/K$ be the quotient map.  If $\phi$ is any $C^\infty$ function on $G$, we define $\phi_\pi$ on $G/K$ by
\begin{equation}\label{E:fcn-projection}
\phi_\pi(gK)=\int_K \phi(gk)\,dk,\qquad g\in G.
\end{equation}
Then the map $\phi\mapsto \phi_\pi$ is a continuous surjection from $\mathcal D(G)$ onto $\mathcal D(G/K)$ (or from $\mathcal E(G)$ onto $\mathcal E(G/K)$). If $\Lambda$ is a distribution on $G/K$, then its pullback is the distribution $\widetilde \Lambda$ on $G$ given by $\widetilde \Lambda(\phi)=\Lambda(\phi_\pi)$. Thus $\Lambda\mapsto \widetilde \Lambda$ is the adjoint of $\phi\mapsto \phi_\pi$. The map $\Lambda\mapsto\widetilde \Lambda$ is continuous and injective, and its range is the closed subspace $\mathcal D'(G)_{R_K}$ of right $K$-invariant distributions on $G$.  If $F\in\mathcal D(G/K)$, then $\Lambda(F)=\widetilde\Lambda(\widetilde F)$.

If $\Lambda$ and $\mu$ are distributions on $G/K$, one with compact support, then their convolution $\Lambda*\mu$ is the distribution on $G/K$ defined by $(\Lambda*\mu)^\sim=\widetilde \Lambda*\widetilde \mu$, where we note that $\widetilde \Lambda*\widetilde \mu$ belongs to $\mathcal D'(G)_{R_K}$. 
Explicitly, for $F\in\mathcal D(G/K)$, we have
\begin{align}\label{E:conv-int}
(\Lambda*\mu)(F)&=(\widetilde\Lambda*\widetilde\mu)(\widetilde F)\nonumber\\
&=\int_{G/K}\int_{G/K} \int_K F(gkhK)\,dk\,d\Lambda(gK)\,d\mu(hK).
\end{align}

If $F$ is a $C^\infty$ function $G/K$ and $\mu$ a distribution on $G/K$, one with compact support, then $F*\mu$ is similarly defined by lifting to $G$, and we have
\begin{equation}\label{E:conv-int2}
F*\mu(gK)=\int_G \int_K F(gkh^{-1}K)\,dk\,d\mu(hK).
\end{equation} 

For a fixed $\mu\in\mathcal E'(G/K)$, let $c_\mu$ denote the right convolution operator $\Lambda\mapsto \Lambda*\mu$ on $\mathcal D'(G/K)$. We will also use $c_\mu$ to denote the right convolution operator $F\mapsto F*\mu$ on the Fr\'echet space $\mathcal E(G/K)$.

\begin{proposition}\label{E:conv-continuous}
$c_\mu$ is a continuous linear operator on $\mathcal D'(G/K)$, as well as on $\mathcal E(G/K)$.
\end{proposition}
\begin{proof}
Let $\mathcal D(G)_K$ and $\mathcal E(G)_K$ denote the closed subspaces of right $K$-invariant functions in $\mathcal D(G)$ and $\mathcal E(G)$, respectively. Then $\mathcal E(G)_K$ is a Fr\'echet space and $\mathcal D(G)_K$ is an $LF$-space (See Lemma \ref{T:lfspace} below). 

The pullback $F\mapsto\widetilde F$, being a continuous bijection from $\mathcal D(G/K)$ onto $\mathcal D(G)_{R_K}$, is therefore a homeomorphism by the Open Mapping Theorem. (This can also be seen directly because the inverse of the pullback is the continuous map $\phi\mapsto \phi_\pi$.) The dual space of $\mathcal D(G)_{R_K}$, which can be identified with $\mathcal D'(G)_{R_K}$, is thus homeomorphic to $\mathcal D'(G/K)$. This means that the pullback $\Lambda\mapsto\widetilde \Lambda$ is a linear homeomorphism from $\mathcal D'(G/K)$ onto $\mathcal D'(G)_{R_K}$.

Hence $\Lambda\mapsto \Lambda*\mu$ can be viewed as the composition of continuous maps
$$
\Lambda\mapsto \widetilde \Lambda\mapsto\widetilde \Lambda*\widetilde \mu=(\Lambda*\mu)^\sim\mapsto \Lambda*\mu.
$$

The proof that $c_\mu$ is continuous on $\mathcal E(G/K)$ is similar.
\end{proof}

If $f$ is a continuous function on $X$, we let $f^\natural$ denote the average $\int_K f^{\ell_k}\,dk$.
 Likewise, if $\mu\in\mathcal E'(G/K)$ , we let $\mu^\natural$ denote the average of $\mu$ with respect to left translations by $K$: $\mu^\natural=\int_K \mu^{\ell_k}\,dk$. Then $\mu^\natural$ is a left $K$-invariant distribution on $G/K$, and the relation \eqref{E:conv-int} shows that $c_\mu=c_{\mu^\natural}$. Thus, in studying mapping properties of $c_\mu$, we do not lose generality by assuming that $\mu$ is left $K$-invariant.

Finally, if $(G,K)$ is a Gelfand pair (as when $(G,K)$ is a Riemannian symmetric pair), we recall that convolution on $G/K$ is commutative on left $K$-invariant functions and distributions.

\section{Convolution Operators on Noncompact Symmetric Spaces}
Many questions about analysis in Euclidean spaces can also be asked about homogeneous spaces, and more specifically symmetric spaces, because of the latter's rich structure. Since the resulting harmonic analysis is necessarily noncommutative and (in the noncompact case) depends to some extent on the Iwasawa and polar decomposition $G=NAK$ and $G=K\overline{A}^+K$, respectively, proofs are often rather different from the ones for $\rr^n$ and involve additional nontrivial  considerations such as Weyl chamber walls.

In the present case, it is natural to try to consider the analogue of Theorem \ref{T:Ehrenpreis60} for Riemannian symmetric spaces of the noncompact type. To settle notation, we will follow that of Helgason's books \cite{DS,GGA,GASS}.


Let $X=G/K$ be a noncompact Riemannian symmetric space, where $G$ is a connected noncompact real semisimple Lie group with finite center, and $K$ is a maximal compact subgroup of $G$. Let $\g$ and $\k$ be the Lie algebras of $G$ and $K$, respectively.  Then  $\g$  has Cartan decomposition $\g=\k+\p$, where $\p$ is the orthogonal complement of $\k$ with respect to the Killing form of $\g$. We let $o$ denote the coset $\{K\}$ of $G/K$, and identify $\p$ with the tangent space $T_oX$.  We endow $X$ with the Riemannian metric induced from the restriction of the Killing form to $\p$.  The map $Y\mapsto \exp Y\cdot o$ is a diffeomorphism of $\p$ onto $X$.

Fix a maximal abelian subspace $\a$ of $\p$ and let $A=\exp \a$. The map $\exp\colon \a\to A$ is a diffeomorphism; let $\log\colon A\to \a$ be its inverse. Let $\Sigma$ be the set of restricted roots of $\g$ with respect to $\a$, and for each $\alpha\in\Sigma$, let $\g_\alpha$ be the corresponding restricted root space, and  let $m_\alpha=\dim \g_\alpha$.
Let $W$ be the Weyl group associated with the root system $\Sigma$.

Fix a Weyl chamber, denoted $\a^+$, in $\a$, and let $\Sigma^+$ be the corresponding system of positive restricted roots. Let $\n=\sum_{\alpha\in\Sigma^+} \g_\alpha$ and $N=\exp \n$.  

The Lie group $G$ has the Iwasawa decompositions $G=KAN=NAK$. For each $g\in G$, in accordance with the decomposition $G=KAN$, we write $g=k(g)\,\exp H(g)\,n(g)$, where $H(g)\in\a$. Likewise, in accordance with the decomposition $G=NAK$, we write $g=n_1(g)\,\exp A(g)\, k_1(g)$, where $A(g)\in\a$. Clearly, $A(g)=-H(g^{-1})$.

A \emph{horocycle} in $X$ is an orbit in $X$ of a conjugate of $N$. The group $G$ acts transitively on the set $\Xi$ of horocycles, so the latter is a homogeneous space of $G$; the isotropy subgroup of $G$ fixing the horocycle $\xi_0=N\cdot o$ is $MN$, where $M$ is the centralizer of $A$ in $K$. Thus $\Xi=G/MN$.

The map $(kM,a)\mapsto ka\cdot\xi_0$ is a diffeomorphism of the product manifold $K/M\times A$ onto $\Xi$. If $\xi=ka\cdot\xi_0$, we say that the coset $kM$ is the \emph{normal} to $\xi$. For convenience, we put $B=K/M$.

For each $x=g\cdot o\in X$ and each $b=kM$, there is a unique horocycle in $X$ containing $x$ with normal $b$. In other words, there is a unique $a\in A$ such that $x\in ka\cdot\xi_0$; let $A(x,b)=\log a$. Then $A(x,b)=-H(g^{-1}k)$. $A(x,b)$ is the multidimensional ``directed distance'' in $X$ from $o$ to the horocycle through $x$ with normal $kM$. ($A(x,b)$ is the analogue, in $G/K$, of the dot product $x\cdot\omega$ in $\rr^n$, for $x\in\rr^n$ and $\omega\in S^{n-1}$, which gives the directed distance from $0\in\rr^n$ to the hyperplane with normal $\omega$ containing $x$.)

The Killing form on $\a$ extends naturally to the complexification  $\a_\cc$, and by duality to the dual space $\a^*$ and ita complexification $\a^*_\cc$. Let $\rho=(1/2)\,\sum_{\alpha\in\Sigma^+} m_\alpha\,\alpha$.

The \emph{Fourier transform} of a function $f\in\mathcal D(X)$ is the function $\widetilde f$ on $\a^*_\cc\times B$ given by
\begin{equation}\label{E:Helg-FT1}
\widetilde f(\lambda,b)=\int_X f(x)\,e^{(-i\lambda+\rho)A(x,b)}\,dx
\end{equation}

Likewise, if $\mu\in\mathcal E'(X)$, its Fourier transform is the function
\begin{equation}\label{E:Helg-FT2}
\widetilde\mu(\lambda,b)=\int_X e^{(-i\lambda+\rho)\,A(x,b)}\,d\mu(x),\qquad (\lambda,b)\in\a^*_\cc\times B.
\end{equation}

In case $\mu$, or $f$, is left $K$-invariant, the Fourier transform is independent of $b$ and 
reduces to the \emph{spherical Fourier transform}, which in the case of $\mu$ is given by
\begin{equation}\label{E:sphertransform1}
\widetilde\mu(\lambda)=\int_X \varphi_{-\lambda}(x)\,d\mu(x),\qquad\lambda\in\a^*_\cc.
\end{equation}
Here $\varphi_\lambda$ is the \emph{spherical function}
$$
\varphi_\lambda(x)=\int_{B} e^{(i\lambda+\rho)\,A(x,b)}\,db,\qquad x\in X.
$$

The inversion formula for the spherical Fourier transform was essentially obtained by Harish-Chandra in 1958 (\cite{HC1958}) and the Paley-Wiener theorem for the Fourier and spherical Fourier transforms on $X$ was obtained by Helgason in 1973 (\cite{He1}); reasonably accessible proofs of these results, now classical, can be found in \cite{GGA}, Ch. IV, and \cite{GASS}, Ch. III.

In particular, let $\mathcal E_K(X)$ denote the subspace of $\mathcal E(X)$ consisting of the left $K$-invariant functions. Its dual space is $\mathcal E'_K(X)$, the subspace of left $K$-invariant distributions in $\mathcal E'(X)$.  Then the spherical Fourier transform $\mu\mapsto\widetilde\mu$ is a linear bijection from $\mathcal E_K'(X)$ onto the space $\mathcal K_W(\a^*_\cc)$ of holomorphic functions on $\a^*_\cc$ of exponential type and of slow increase in $\a^*$. 


 For a fixed $\mu\in\mathcal E_K'(X)$, let $c_\mu\colon \mathcal E(X)\to\mathcal E(X)$ be the convolution operator $c_\mu(f)=f*\mu$. In a previous paper (\cite{CGK2017}, Theorem 5.1), it was proved that if $\widetilde\mu$ is slowly decreasing, then  $c_\mu\colon\mathcal E(X)\to\mathcal E(X)$ is surjective.

In this paper, one of our main results is the converse assertion, which we state below. 

\begin{theorem}\label{T:mainthm}
Suppose that $\mu\in\mathcal E'_K(X)$. If $c_\mu\colon\mathcal E(X)\to\mathcal E(X)$ is surjective, then $\widetilde\mu$ is a slowly decreasing function on $\a^*_\cc$.
\end{theorem}


The main proof of Theorem \ref{T:mainthm} will occur in Section 4. The idea is to transfer the analysis to $\a$ by means of the Abel transform. In order to start the process, we first recall a few facts about the horocycle Radon transform and the Abel transform on compactly supported distributions on $X$. 

Let $F\in\mathcal E(\a)$. For each $b=kM\in B$, let $F^b$ be the \emph{horocycle plane wave} $F^b(x)=F(A(x,b))$. (It is called a horocycle plane wave since it is constant on horocycles with normal $b$.) The function $\Phi_F$ on $X\times B$ given by $\Phi_F(x,b)=F^b(x)$ is clearly $C^\infty$, and the linear map $F\mapsto \Phi_F$ from $\mathcal E(\a)$ to $\mathcal E(X\times B)$ is therefore a continuous map of Fr\'echet spaces, being the pullback of the smooth map $(x,b)\mapsto A(x,b)$ from $X\times B$ to $\a$. 

In particular, for each $b\in B$, the map $F\mapsto F^b$ is continuous from $\mathcal E(\a)$ to $\mathcal E(X)$. The adjoint of this map is the \emph{horocycle Radon transform} $R_b$ from $\mathcal E'(X)$ to $\mathcal E'(\a)$. Explicitly, for fixed $b\in B$ and $\mu\in\mathcal E'(X)$, the Radon transform $R_b\,\mu$ is the (compactly supported) distribution on $\a$ given by
\begin{equation}\label{E:radontransform1}
R_b\,\mu(F)=\mu(F^b),\qquad \qquad\text{for all }F\in\mathcal E(\a).
\end{equation}
Thus 
$$
R_b\mu(F)=\int_X F(A(x,b))\,d\mu(x),\qquad F\in\mathcal E(\a).
$$
Since $\|A(x,b)\|\leq d(o,x)$ for all $x\in X$ and all $b\in B$, it is clear that if $\mu$ has support in the closed ball $\overline B_R(o)\subset X$, then $R_b\,\mu$ has  support in the closed ball $\{H\in\a\,\colon\,\|H\|\leq R\}$.

If $\widetilde \mu$ denotes the Fourier transform of $S$, then we have the \emph{projection-slice theorem}:
\begin{align}
\widetilde \mu(\lambda,b)&=\int_X e^{(-i\lambda+\rho)A(x,b)}\,d\mu(x)\nonumber\\
&=\left(e^\rho\,R_b\,\mu\right)^*(\lambda).\label{E:projslice}
\end{align}
where $S\mapsto S^*$ denotes the Euclidean Fourier transform on $\a$.

Note that $R_b\,\mu^{\tau(k)}=R_{k^{-1}\cdot b}\,\mu$ for all $k\in K$. Thus if $\mu\in\mathcal E_K'(X)$, we see that $R_b\,\mu=R_{b_0}\,\mu$ for all $b\in B$, where $b_0=eM$.

 If $\mu\in\mathcal E_K'(X)$, we define its \emph{Abel transform} $\mathcal A\mu\in\mathcal E'(\a)$  by
\begin{equation}\label{E:abeltr}
\mathcal A\mu=e^\rho\,R_{b_0}\mu.
\end{equation}
Since $\mu$ is left $K$-invariant, $b_0$ can be replaced by any $b\in B$. For $\mu\in\mathcal E'_K(X)$, the projection-slice theorem \eqref{E:projslice} becomes
\begin{equation}\label{E:abelslice}
\widetilde\mu(\lambda)=(\mathcal A\mu)^*(\lambda),\qquad \lambda\in\a^*_\cc.
\end{equation}

Let $\mathcal E_W(\a)$ denote the subspace consisting of all $W$-invariant elements of $\mathcal E(\a)$. Its dual space can be identified with the space $\mathcal E'_W(\a)$ consisting of all $W$-invariant elements of $\mathcal E'(\a)$.

The Paley-Wiener Theorem for $K$-invariant distributions  implies that the Abel transform is a linear bijection from $\mathcal E'_K(X)$ onto $\mathcal E'_W(\a)$ and in fact we have a commutative diagram of linear bijections
\begin{equation}\label{E:abelcommutative}
\begin{tikzcd}
\mathcal E'_K(X)\arrow[rr,"\mathcal A"] \arrow[dr,"\sim"] && \mathcal E'_W(\a) \arrow[dl,swap,"*"] \\
& \mathcal H_W(\a^*_\cc)
\end{tikzcd}
\end{equation}
Here $\mathcal H_W(\a^*_\cc)$ is the vector space of slowly increasing $W$-invariant holomorphic functions on $\a^*_\cc$ of exponential type. 



For each $F\in\mathcal E(\a)$, let $TF$ denote the $K$-invariant function on $X$ given by $TF=((e^\rho\, F)^b)^\natural$, where $b$ is any element of $B$. It is easy to see that
$$
TF(x)=\int_B e^{\rho(A(x,b))}\,F(A(x,b))\,db,\qquad x\in X.
$$

\begin{proposition}\label{T:bijection1}
The map $T$ is a linear bijection from $\mathcal E_W(\a)$ onto $\mathcal E_K(X)$.
\end{proposition}

This bijection is of course different from the well-known bijection given by the restriction map from $\mathcal E_K(X)\approx \mathcal E_K(\p)$ onto $\mathcal{E}_W(\a)$.

\begin{proof}
The map $T$ is clearly a continuous linear map of Fr\'echet spaces, and for any $F\in \mathcal E_W(\a)$ (in fact for any $F\in\mathcal E(\a)$) it is clear that $TF\in \mathcal E_K(X)$. Let us now prove that the adjoint $T^*$ coincides with the Abel transform $\mathcal A\colon \mathcal E'_K(X)\to\mathcal E'_W(\a)$.  Suppose that $\mu\in\mathcal E_K'(X)$. Then for any $F\in \mathcal E_W(\a)$ we have
\begin{align*}
T^*\mu(F)&=\mu(TF)\\
&=\int_X\int_B e^{\rho(A(x,b))}\,F(A(x,b))\,db\,d\mu(x)\\
&=\int_B\int_X e^{\rho(A(x,b))}\,F(A(x,b))\,d\mu(x)\,db\\
&=\int_B\mu((e^\rho\,F)^b)\,db\\
&=\int_B R_b\mu\,(e^\rho\,F)\,db\\
&=(\mathcal A\mu)\,(F),
\end{align*}
the last inequality resulting from the identity $e^\rho\,R_b\mu=\mathcal A\mu$, for all $b\in B$.

Since the Abel transform $\mathcal A\colon \mathcal E_K'(X)\to\mathcal E_W'(\a)$ is a bijection, it is injective and has closed range (namely all of $\mathcal E_W'(\a)$). It follows from Proposition \ref{T:frechet} below that $T$ is a bijection.
\end{proof}

In the last part of the proof above we have used the following fact about linear mappings of Fr\'echet spaces.

\begin{proposition}\label{T:frechet}
	Let $E$ and $F$ be Fr\'echet spaces. A continuous linear map $\Phi\colon E\to F$ is surjective if and only if the adjoint map $T^*\colon F'\to E'$ is injective and has weak* closed range in $E'$.
\end{proposition}
For a proof, see Theorem 7.7, Ch. IV in \cite{Sch}. See also Theorem 3.7, Ch. I in \cite{GASS} for a generalization. 

 Since $\mathcal E_W(\a)$ and $\mathcal E_K(X)$ are Fr\'echet spaces, the Open Mapping Theorem implies that $T$ is a homeomorphism, and it follows that its adjoint, the Abel transform $\mathcal A\colon\mathcal E'_K(X)\to\mathcal E'_W(\a)$, is also a homeomorphism (for the strong or weak* topologies). Using either of these dual topologies, we can therefore topologize $\mathcal{H}_W(\a^*_\cc)$ so that all maps in the diagram \eqref{E:abelcommutative} are homeomorphisms.


If $\Lambda\in\mathcal E'(\a)$, let $c_\Lambda$ denote the Euclidean convolution operator on $\mathcal E(\a)$ (respectively $\mathcal E'(\a)$) given by
$F\mapsto F*\Lambda$ (respectively $\Psi\mapsto\Psi*\Lambda$). Yet again abusing notation slightly, we will
 denote by $c_\Lambda$ the convolution operator on $\mathcal E(\Xi)$ given by
$$
c_\Lambda(\varphi)(g\cdot\xi_0)=\int_A \varphi(g\,\exp (-H)\cdot\xi_0)\,d\Lambda(H),\qquad \varphi\in\mathcal E(\Xi).
$$
Since the lift to $G$ of the right hand side is smooth, one sees that $c_\Lambda\varphi\in\mathcal E(\Xi)$.

\begin{proposition}\label{T:commutative1}
Suppose that $\mu\in\mathcal E_K'(X)$. Then the following diagram commutes:
\begin{equation}\label{E:commdiag1}
\begin{tikzcd}
\mathcal E_W(\a) \arrow{r}{c_{\mathcal A\mu}} \arrow[swap]{d}{T} & \mathcal E_W(\a) \arrow{d}{T} \\
\mathcal E_K(X) \arrow{r}{c_\mu}& \mathcal E_K(X)
\end{tikzcd}
\end{equation}
\end{proposition}
\begin{proof}
The commutativity of the diagram \eqref{E:commdiag1} can be proven by direct computation, but since all the maps in it are continuous, it is easier to prove it by showing that the diagram consisting of the adjoint maps is commutative:
\begin{equation}\label{E:adjcomm}
\begin{tikzcd}
\mathcal E'_W(\a)& \arrow[l,"c^*_{\mathcal A\mu}", swap]\mathcal E'_W(\a)\\
\arrow[u,"T^*"]\mathcal E'_K(X)&\arrow[l,"c_\mu^*",swap]\mathcal E'_K(X)\arrow[u,"T^*", swap] 
\end{tikzcd}
\end{equation}
But as shown in the proof of Proposition \ref{T:bijection1}, the adjoint $T^*$ coincides with the Abel transform $\mathcal A$.
In addition we have $c_\mu^*=c_{\widecheck\mu}$, where $\widecheck\mu$ is the $K$-invariant distribution on $X$ whose pullback to $G$ is the reflection of the pullback of $\mu$ with respect to the inversion $g\mapsto g^{-1}$. 

Finally we have $c^*_{\mathcal A\mu}=c_{(\mathcal A \mu)^\vee}$, where $(\mathcal A\mu)^\vee$ is the reflection of the distribution 
$\mathcal A\mu$ with respect to the map $H\mapsto -H$ of $\a$. Thus we wish to prove the commutativity of the diagram
\begin{equation}\label{E:commdiagE'}
\begin{tikzcd}
\mathcal E'_W(\a)& \arrow[l,"c_{(\mathcal A\mu)^\vee}", swap]\mathcal E'_W(\a)\\
\arrow[u,"\mathcal A"]\mathcal E'_K(X)&\arrow[l,"c_{\widecheck\mu}",swap]\mathcal E'_K(X)\arrow[u,"\mathcal A", swap] 
\end{tikzcd}
\end{equation}
which amounts to showing that
$$
\mathcal AS*(\mathcal A\mu)^\vee=\mathcal A(S*\widecheck \mu)
$$
for all $S\in\mathcal E'_K(X)$. But the Euclidean Fourier transforms of both sides in $\a$ equal $\widetilde S(\lambda)\widetilde\mu(-\lambda)$,  completing the proof.
\end{proof}


Now if $\mu$ is $K$-invariant the relation $(f*\mu)^\sim(\lambda,b)=\widetilde f(\lambda,b)\,\widetilde\mu(\lambda)$ and the projection-slice theorem shows that
\begin{equation}\label{E:conv2}
R(f*\mu)=Rf*\mathcal A\mu
\end{equation}
for all $f\in\mathcal D(X)$, where the convolution on the right is taken over the $A$ factor in $\Xi=K/M\times A$; explicitly,
$$
Rf*\mathcal A\mu(b,\exp H)=\int_{\a} Rf(b,\exp(H-H'))\,d(\mathcal A\mu)(H').
$$
 
The \emph{dual transform} $R^*$ from functions on $\Xi$ to functions on $X$ is defined by
\begin{align}
R^*\varphi(g\cdot o)&=\int_K \varphi(gk\cdot\xi_0)\,dk\nonumber\\
&=\int_B \varphi(b,\exp(A(x,b)))\,e^{2\rho(A(x,b))}\,db\qquad (\varphi\in\mathcal E(\Xi)).\label{E:dualtransform}
\end{align}
It is known \cite{GASS} that the linear map $R^*\colon\mathcal E(\Xi)\to\mathcal E(X)$ is surjective.

Let $b_0=eM\in B$. If $\mu\in \mathcal E'_K(X)$, let $\widehat{\mu}=R_{b_0}\mu=e^{-\rho}\,\mathcal A\mu$. Then $\widehat\mu\in\mathcal E'(\a)$. 


\begin{proposition}\label{T:commutative2}
If $\mu\in\mathcal E_K'(X)$, then the following diagram commutes:
$$
\begin{tikzcd}
\mathcal E(\Xi) \arrow{r}{c_{\widehat\mu}} \arrow[swap]{d}{R^*} & \mathcal E(\Xi) \arrow{d}{R^*} \\
\mathcal E(X) \arrow{r}{c_\mu}& \mathcal E(X)
\end{tikzcd}
$$
\end{proposition}

\noindent\textit{Remark:} This proposition is in some respects a refinement of Proposition \ref{T:commutative1}. It is also the symmetric space analogue of the following well-known  formula for the $d$-plane Radon transform $R$ on $\rr^n$:
$$
R^*(Rf*\varphi)=f*R^*\varphi,
$$
which holds for all $f\in\mathcal D(\rr^n),\,\varphi\in\mathcal E(G(d,n))$, where $G(d,n)$ is the manifold of unoriented $d$-planes in $\rr^n$. The convolution $Rf*\varphi$ is carried out along the fibers of the vector bundle $G(d,n)\to G_{d,n}$, where $G_{d,n}$ is the Grassmann manifold of $d$-dimensional subspaces of $\rr^n$.

\begin{proof} Let $\varphi\in\mathcal E(\Xi)$, and let $\Phi$ be its lift to $G$, so that $\Phi(g)=\varphi(g\cdot\xi_0)$. Likewise, let $U$ be the lift of the distribution $\mu$ to $G$, given by
$$
U(F)=\mu(F_\pi)
$$
for all $F\in\mathcal E(G))$, where $F_\pi$ is as in \eqref{E:fcn-projection}.
It is clear that $U$ is a $K$-biinvariant element of $\mathcal E'(G)$.

For any $g_0\in G$, we have
\begin{align*}
R^*(c_{\widehat\mu}(\varphi))(g_0K)&=\int_K (c_{\widehat\mu}\varphi)(g_0k\cdot\xi_0)\,dk\\
&=\int_K\int_{\a} \varphi(g_0k\,\exp (-H)\cdot\xi_0)\,d\widehat\mu(H)\,dk\\
&=\int_K \int_X \varphi(g_0k\,\exp(-A(x,eM))\cdot\xi_0)\,d\mu(x)\,dk\\
&=\int_K \int_G \varphi(g_0k\,\exp H(g^{-1})\cdot\xi_0)\,dU(g)\,dk\\
&=\int_G\int_K \varphi(g_0\,k\,\exp H(g^{-1})\cdot\xi_0)\,dk\,dU(g)\\
\intertext{(by Fubini's Theorem for distributions)}
&=\int_G \int_K \varphi(g_0\, k\,k(g^{-1})\,\exp H(g^{-1})\cdot\xi_0)\,dk\,dU(g)\\
\intertext{(since $K$ is unimodular)}
&=\int_G \int_K \varphi(g_0\, k\, g^{-1}\cdot\xi_0)\,dk\,dU(g)\\
&=\int_K \int_G \varphi(g_0\, k\, g^{-1}\cdot\xi_0)\,dU(g)\,dk\\
&=\int_K \int_G \varphi(g_0\, (gk^{-1})^{-1}\cdot\xi_0)\,dU(g)\,dk\\
&=\int_K \int_G \varphi(g_0\, g^{-1}\cdot\xi_0)\,dU(g)\,dk\\
\intertext{(since $U$ is right $K$-invariant)}\\
&=\int_G \varphi(g_0\, g^{-1}\cdot\xi_0)\,dU(g)\\
&=\int_K \int_G \varphi(g_0\, (k^{-1}g)^{-1}\cdot\xi_0)\,dU(g)\,dk\\
\intertext{(since $U$ is left $K$-invariant)}
&=\int_K \int_G \varphi(g_0\, g^{-1}\,k\cdot\xi_0)\,dU(g)\,dk\\
&=\int_G \int_K \varphi(g_0\, g^{-1}\,k\cdot\xi_0)\,dk\,dU(g)\\
&=\int_G R^*\varphi(g_0g^{-1}K)\,dU(g)\\
&=(R^*\varphi * \mu)(g_0K)\\
&=c_\mu(R^*\varphi)(g_0K),
\end{align*}
proving the proposition.
%

\end{proof}

\section{The Slow Decrease Property on $\rr^n$.}
In this section, we will prove Theorem \ref{T:mainthm} in earnest. 

Suppose that $c_\mu\colon\mathcal E(X)\to\mathcal E(X)$ is surjective. Then for any $\psi\in \mathcal E_K(X)$, there is a $\varphi\in\mathcal E(X)$ for which $\varphi*\mu=\psi$. Replacing $\varphi$ by $\varphi^\natural$, we may assume that $\varphi\in\mathcal E_K(X)$. Thus $c_\mu$ maps $\mathcal E_K(X)$ onto $\mathcal{E}_K(X)$.

Now according to Proposition \ref{T:bijection1}, the linear map $T\colon\mathcal E_W(\a)\to\mathcal E_K(X)$ is a linear bijection.  Proposition \ref{T:commutative1} therefore shows that the convolution operator $c_{\mathcal A \mu}\colon\mathcal E_W(\a)\to\mathcal E_W(\a)$ is surjective. We claim that this implies that the distribution $\mathcal A\mu$ on $\a$ is invertible, that is, the holomorphic function $(\mathcal A\mu)^*(\lambda)=\widetilde\mu(\lambda)$ is slowly decreasing, which will of course prove Theorem \ref{T:mainthm}. Our claim will be proved in Theorem \ref{T:W-surjectivity} below. (This theorem will also imply that $c_{\mathcal A_\mu}\colon\mathcal E(\a)\to\mathcal E(\a)$ is surjective.)


The problem is now Euclidean and turns out not to depend on any specific properties of the Weyl group, so for convenience we'll replace $\a$ by $\rr^n$ and assume that $W$ is just a finite subgroup of the orthogonal group $O(n)$.

There is at least one point $x\in\rr^n$ such that the isotropy subgroup $W_x$ of $W$ at $x$ is $\{e\}$. 
In fact, for any $w\in W$, let $(\rr^n)^{w}$ be the subspace of $\rr^n$ consisting of points fixed by $w$. If $w\neq I$, then $(\rr^n)^{w}$ is a proper subspace of $\rr^n$. Hence the finite union
$$
\bigcup_{w\neq I} (\rr^n)^{w}
$$
is not all of $\rr^n$, and we may take $x$ to be any point in the complement of the set above. Note that the complement of the set above is an open cone, so we can scale $x$ so as to be close to the origin if necessary. The orbit $W\cdot x$ then consists of $|W|$ distinct points.


The spaces $\mathcal E_W(\rr^n)$ and $\mathcal E'_W(\rr^n)$ (consisting of $W$-invariant smooth functions and compactly supported distributions, respectively) are closed in $\mathcal E(\rr^n)$ and $\mathcal E'(\rr^n)$, respectively.


As mentioned earlier, our objective is to prove the following result.

\begin{theorem}\label{T:W-surjectivity}
Let $\Lambda\in\mathcal E'_W(\rr^n)$ and suppose that the convolution operator $c_\Lambda\colon \mathcal E_W(\rr^n)\to\mathcal E_W(\rr^n)$ is surjective. Then $\Lambda$ is invertible.
\end{theorem}

The conclusion above is not immediate from Theorem \ref{T:Ehrenpreis60} (b), since all we have is that convolution with $\Lambda$ is surjective on $W$-invariant functions. (The converse is obvious: if $\Lambda$ is invertible, then $c_\Lambda\colon\mathcal E_W(\rr^n)\to\mathcal E_W(\rr^n)$ is onto.) 

To prove Theorem \ref{T:W-surjectivity}, we'll take a ``hard'' analytical approach, although a representation-theoretic approach may also be possible. See Remark \ref{R:rep-theoretic} at the end of this section.

Now in addition to the conditions in Theorem \ref{T:Ehrenpreis60}, there are several other equivalent formulations for a distribution
 $\Lambda\in\mathcal E'(\rr^n)$ to be invertible. These can be found in \cite{Ehrenpreis1960} as well as in Chapter 16 of H\"ormander's book \cite{Ho2}. Among these, we will use the following.


\begin{lemma}\label{T:slowdecrease1} (\cite{Ho2}, Theorems 16.3.9 and 16.3.10.)
For a distribution $\Lambda\in\mathcal E'(\rr^n)$, the following are equivalent:
\begin{enumerate}
\item $\Lambda$ is not invertible.
\item For every $x\in\rr^n$ one can find $S\in\mathcal E'(\rr^n)\cap (C(\rr^n)\setminus C^1(\rr^n))$ with $\text{sing supp }S=\{x\}$ and $\Lambda*S\in C^\infty(\rr^n)$.
\item There exists an $S\in\mathcal E'(\rr^n)$ such that $\Lambda*S\in C^\infty(\rr^n)$ but $S\notin C^\infty(\rr^n)$.
\end{enumerate} 
\end{lemma}

The lemma above  can be made more or less $W$-invariant as follows.

\begin{lemma}\label{T:W-invertible-cond}  Fix any point $x_0\in\rr^n$ such that $W_{x_0}=\{e\}$. A distribution $\Lambda\in\mathcal E'_W(\rr^n)$ is invertible if and only if for all $S\in\mathcal E'_W(\rr^n)$ with $\text{sing supp }S\subset W\cdot x_0$, $\Lambda*S\in C^\infty(\rr^n)$ implies that $S\in C^\infty(\rr^n)$.
\end{lemma}

\begin{proof}
Suppose that $\Lambda$ is invertible. If $\Lambda*S\in C^\infty(\rr^n)$, then (3) in Lemma \ref{T:slowdecrease1} implies that $S\in C^\infty(\rr^n)$.


Conversely, suppose that $\Lambda$ is not invertible. Then (2) in Lemma \ref{T:slowdecrease1} implies that there is an $S_1\in\mathcal E'(\rr^n)\cap (C(\rr^n)\setminus C^1(\rr^n))$, with $\text{sing supp }S_1=\{x_0\}$, such that $\Lambda*S_1\in C^\infty(\rr^n)$. 

Let $S=\sum_{w\in W} w\cdot S_1$. Then $S$ is not $C^1$ at any orbit point $w\cdot x_0$, since $w\cdot S_1$ is not $C^1$ at $w\cdot x_0$, but $w'\cdot S_1$ is $C^\infty$ in a neighborhood of $w\cdot x_0$, for any $w'\neq w$. Thus $S\in\mathcal E'_W(\rr^n)$, $\text{sing supp }S=W\cdot x_0$, and
$$
\Lambda*S=\sum_{w\in W} w\cdot (\Lambda*S_1)\in C^\infty(\rr^n).
$$
\end{proof}

For any $s\in\rr$, let $H_{(s)}$ denote the Sobolev space with index $s$, equipped with norm
$$
\|u\|_{(s)}=\left((2\pi)^{-n}\,\int_{\rr^n} (1+\|\xi\|^2)^s\,|u^*(\xi)|^2\,d\xi\right)^{1/2}
$$ 
Let $L$ be the Laplace operator on $\rr^n$. 

\begin{lemma}\label{T:sobolev1} For any $N\in\zz^+$, there exists a constant $C$  such that
$$
\sup |L^k u(x)|\leq C\,\|u\|_{(2N+n)},
$$
for all $u\in \mathcal S(\rr^n)$ and all $k\leq N$.
\end{lemma}
\begin{proof}
For any  $s>n/2$ and  $k\leq N$, the Fourier inversion formula and the H\"older Inequality imply that
\begin{align*}
\sup |L^k(x)|&\leq (2\pi)^{-n}\,\int_{\rr^n} (1+\|\xi\|^2)^k\,|u^*(\xi)|\,d\xi\\
&\leq (2\pi)^{-n}\,\int_{\rr^n} (1+\|\xi\|^2)^{-s/2}\,(1+\|\xi\|^2)^{N+s/2}\,|u^*(\xi)|\,d\xi\\
&\leq C\,\|u\|_{(2N+s)},
\end{align*}
where $C=(2\pi)^{-n/2}\,(\int (1+\|\xi\|^2)^{-s}\,d\xi)^{1/2}$. Letting $s=n$ gives us the result.

\end{proof}

The following lemma, which is modeled on \cite{Ho2}, Theorem 16.5.1, will imply Theorem \ref{T:W-surjectivity}.

\begin{lemma}\label{T:W-inv-surj}
Let $\Lambda\in\mathcal E'_W(\rr^n)$. Suppose that for any $W$-invariant $f\in C_c^\infty(\rr^n)$, there is a $W$-invariant $S_f\in\mathcal D'(\rr^n)$ such that $\Lambda*S_f=f$. Then $\Lambda$ is invertible.
\end{lemma} 
\begin{proof}
If $v\in C_c^\infty(\rr^n)$ is $W$-invariant, then for $f$ and $S_f$ as above, we have
\begin{equation}\label{E:convol1}
\int_{\rr^n} f(x)\,v(x)\,dx=S_f(\widecheck\Lambda*v),
\end{equation}
where $\widecheck\Lambda$ is the reflection of $\Lambda$ through the origin.  We claim that for any closed ball $\overline B_R$ of radius $R>0$ centered at the origin, there are constants $C$ and $N$ (depending on $R$) such that
\begin{equation}\label{E:estimate1}
\left|\int_{\rr^n} f\,v\,dx\right|\leq C\sum_{k\leq N} \sup |L^k f|\cdot\sum_{m\leq N} \sup|L^m(\widecheck\Lambda*v)|,
\end{equation}
for all $W$-invariant $f$ and $v$ in $C_c^\infty(\overline B_R)$.

To prove \eqref{E:estimate1}, let $F$ be the vector space $C_c^\infty(\overline B_R)$ equipped with the topology defined by the seminorms $f\mapsto \sup |D^\alpha f|$, for all multiindices $\alpha$, and let $V$ be the vector space $C_c^\infty(\overline B_R)$ equipped with the topology defined by the seminorms $v\mapsto \sup |D^\beta(\widecheck\Lambda*v)|$, for all multiindices $\beta$. (Note that $\widecheck\Lambda*v$ is generally supported in a ball larger than $\overline B_R$, but this does not matter.)  Then $F$ and $V$ are metrizable topological vector spaces and $F$ is complete. 

The topologies on $F$ and $V$ are equally well defined by the seminorms
\begin{equation}\label{E:seminorm}
f\mapsto \sup|L^k f|\quad\text{and}\quad v\mapsto \sup |L^m(\widecheck\Lambda*v)|,\qquad k,m\in\zz^+.
\end{equation}
This is because $\|\xi\|^{2k}\,|f^*(\xi)|\leq \text{vol}\,(B_R)\cdot \sup|L^k f|$ for any $\xi\in\rr^n$ and for any $k$. Hence for any multiindex $\beta$, 
we can choose $m\geq |\beta|/2$ to obtain
\begin{align*}
|D^\beta f(x)|&\leq (2\pi)^{-n}\,\int_{\rr^n} |\xi^\beta|\,|f^*(\xi)|\,d\xi\\
&\leq (2\pi)^{-n}\,\int_{\rr^n} (1+\|\xi\|^2)^m\, |f^*(\xi)|\,d\xi\\
&= (2\pi)^{-n}\,\int_{\rr^n} (1+\|\xi\|^2)^{-n}\,(1+\|\xi\|^2)^{m+n}\,|f^*(\xi)| \,d\xi\\
&\leq (2\pi)^{-n}\,\text{vol}\,(B_R)\cdot\int_{\rr^n} (1+\|\xi\|^2)^{-n}\,d\xi\cdot \sup |(1+L)^{m+n}\,f|\\
&=C\,\sup |(1+L)^{m+n}\,f|.
\end{align*}

We will use the seminorms \eqref{E:seminorm} for $F$ and $V$. Let $F_W$ and $V_W$ be the subspaces of $F$ and $V$ consisting of $W$-invariant elements.  Then $F_W$ and $V_W$ are closed subspaces of $F$ and $V$ with topologies given by the seminorms \eqref{E:seminorm}. Now consider the bilinear form on $F_W\times V_W$ given by
\begin{equation}\label{E:bilinear}
\langle f,v\rangle=\int_{\rr^n} f\,v\,dx.
\end{equation}
For each fixed $v$, the bilinear form \eqref{E:bilinear} is trivially continuous with respect to $f$, and for each fixed $f$, \eqref{E:convol1} implies that the form is continuous with respect to $v$. Since $F_W$ is an $F$-space and $V_W$ is metrizable, the Banach-Steinhaus Theorem implies that the bilinear form \eqref{E:bilinear} is jointly continuous on $F_W\times V_W$, proving \eqref{E:estimate1}. (See \cite{RudinFA}, Theorem 2.17.)

Now suppose that $x_0$ is a point such that $W_{x_0}=\{e\}$. Since $\Lambda$ is invertible if and only if $\widecheck\Lambda$ is invertible, Lemma \ref{T:W-invertible-cond} shows that it suffices to prove that if $S\in\mathcal E'_W(\rr^n)$ with $\text{sing supp }S\subset W\cdot x_0$, then $\widecheck\Lambda*S\in C^\infty(\rr^n)$ implies that $S\in C^\infty(\rr^n)$.

So fix such an $S$. Then fix an open ball $B_R$ containing the support of $S$, and then consider the subspaces $F_W$ and $V_W$ of $C_c^\infty(\overline B_R)$ as above.   Let $\chi$ be any nonnegative radial compactly supported $C^\infty$ function with $\int_{\rr^n}\chi\,dx=1$, and set $\chi_\delta(x)=\delta^{-n}\,\chi(x/\delta)$. For small $\delta$, the convolution
$$
S_\delta=S*\chi_\delta
$$
is an element of $V_W$. Moreover, for any $m\in\zz^+$, we have
\begin{equation}\label{E:estimate2}
\sup |L^m(\widecheck\Lambda*S_\delta)|= \sup |L^m(\widecheck\Lambda*S)*\chi_\delta|\leq \sup |L^m(\widecheck\Lambda*S)|<\infty.
\end{equation}
Let $\ell\in\zz^+$. Then for small $\delta$, applying the estimate \eqref{E:estimate1} to $L^\ell S_\delta$ and using \eqref{E:estimate2} above, we obtain
\begin{align}
\left|\int_{\rr^n} f\,(L^\ell(S*\chi_\delta))\,dx\right|&\leq C\,\sum_{m\leq N} \sup |L^m(\widecheck\Lambda*L^\ell(S*\chi_\delta))|\cdot \sum_{k\leq N}\sup |L^k f|\nonumber\\
&\leq C\,\sum_{m\leq N} \sup |L^{\ell+m}(\widecheck\Lambda*S)|\cdot \sum_{k\leq N} \sup |L^k f|\label{E:estimate3}\\
&=C_\ell\,\sum_{k\leq N} \sup |L^k f|,\nonumber
\end{align}
for all $f\in F_W$, where $C_\ell$ is the product of the first two factors in \eqref{E:estimate3}.
Lemma \ref{T:sobolev1} then shows that there is  a constant $D_\ell$  such that
$$
\left|\int_{\rr^n} f\,(L^\ell S*\chi_\delta)\,dx\right|\leq D_\ell\,\|f\|_{(2N+n)},
$$
for all $f\in F_W$ and all small $\delta>0$. Letting $\delta\to 0$, we obtain
\begin{equation}\label{E:sobolev2}
|L^\ell S\, (f)|\leq D_\ell \, \|f\|_{(2N+n)},
\end{equation}

for all $f\in F_W$. 

Now let $g\in F$.  Since $L^\ell S$ is $W$-invariant, we can apply the estimate \eqref{E:sobolev2} to the average
$f=(\sum_{w\in W} g^w)/|W|\in F_W$ to obtain
\begin{align}
|L^\ell S\,(g)|&=|L^\ell S\,(f)|\nonumber\\
&\leq D_\ell \|f\|_{(2N+n)}\nonumber\\
&\leq \frac{D_\ell}{|W|}\,\sum_{w\in W} \|g^w\|_{(2N+n)}\nonumber\\
&= D_\ell \,\|g\|_{(2N+n)}.\label{E:sobolev3}
\end{align}

Now fix a function $\psi\in C_c^\infty(\overline B_R)$ that is identically $1$ in a neighborhood of $\text{supp }S$. Then for all $g\in\mathcal S(\rr^n)$, \eqref{E:sobolev3} shows that
\begin{align*}
|L^\ell S(g)|&=|L^\ell S(\psi g)|\\
&\leq D_\ell \,\|\psi\,g\|_{(2N+n)}\\
&\leq C'\,\|g\|_{(2N+n)},
\end{align*}
for all $g\in\mathcal S(\rr^n)$, where $C'$ depends on $\ell,\,R$, and $\psi$. (Here we have used the fact that $g\mapsto \psi\,g$ is bounded on $H_{(s)}$, for all $s\in\rr$. See, e.g., \cite{Folland}, Proposition 6.12.) Since $\mathcal S(\rr^n)$ is dense in $H_{(2N+n)}$, $L^\ell S$ extends uniquely to a bounded linear functional on $H_{(2N+n)}$. Hence $L^\ell S\in H_{(-2N-n)}$ for all $\ell$, and this shows that $S\in C^\infty(\rr^n)$.

This finishes the proof of Lemma \ref{T:W-inv-surj}, and completes the proof of Theorem \ref{T:W-surjectivity}, as well as our main result Theorem \ref{T:mainthm}.
\end{proof}

\begin{remark}\label{R:rep-theoretic}
 It may also be possible to prove that $\Lambda*\mathcal{E}(\rr^n)=\mathcal E(\rr^n)$ as follows. Both $\mathcal E(\rr^n)$ and $\mathcal E'(\rr^n)$ are (finite) direct sums of $W$-isotypic components:
\begin{equation}\label{E:isotypic}
\mathcal E(\rr^n)=\sum_{\pi\in \widehat W} \mathcal E_\pi(\rr^n),\qquad \mathcal E'(\rr^n)=\sum_{\pi\in \widehat W} \mathcal E'_\pi(\rr^n),
\end{equation}
where $\widehat{W}$ is the unitary dual of $W$. Here of course, $\mathcal E_W(\rr^n)$ and $\mathcal E'_W(\rr^n)$ are the isotypic components corresponding to the trivial representation of $W$. If one can prove that
\begin{equation}\label{E:Ktype}
\mathcal E_W(\rr^n)*\mathcal E'_\pi(\rr^n)=\mathcal E_\pi(\rr^n),
\end{equation}
then it will follow that
$$
\Lambda*\mathcal E_\pi(\rr^n)=\Lambda*\mathcal E_W(\rr^n)*\mathcal E'_\pi(\rr^n)=\mathcal E_W(\rr^n)*\mathcal E'_\pi(\rr^n)=\mathcal{E}_\pi(\rr^n),
$$
and, by \eqref{E:isotypic}, this implies that $\Lambda*\mathcal E(\rr^n)=\mathcal E(\rr^n)$.   It can be shown that $\mathcal E_W(\rr^n)*\mathcal E'_\pi(\rr^n)$ is dense in $\mathcal E_\pi(\rr^n)$. (In fact, one only needs to convolve $\mathcal E_W(\rr^n)$ with finitely many distributions in $\mathcal E'_\pi(\rr^n)$ to achieve density.) However, as of this writing the authors could not obtain a proof of the equality in \eqref{E:Ktype}.
\end{remark}

\section{Convolutions with General Distributions}

We now extend our results to general distributions $\mu\in \mathcal{E'}(X)$.  Recall from Section 2 that  we defined $\mu^\natural$ to be the average of the left $K$-translates of $\mu$. As observed in that section, we have $c_\mu=c_{\mu^\natural}$.

%

\begin{theorem}\label{T:surjectivity2}
	Let $\mu\in\mathcal E'(X)$. Then the right convolution operator $c_\mu\colon f\mapsto f*\mu$ from $\mathcal E(X)$ to $\mathcal E(X)$ is surjective if and only if the holomorphic function on $\a^*_\cc$ given by
	$$
	\lambda\mapsto\int_B \widetilde\mu(\lambda,b)\,db
	$$
	is slowly decreasing.
\end{theorem}
\begin{proof}
Since $c_\mu=c_{\mu^\natural}$, $c_\mu$ is surjective if and only if $c_{\mu^\natural}$ is surjective. But by Theorem \ref{T:mainthm} and Theorem 5.1 in \cite{CGK2017}, $c_{\mu^\natural}$ is surjective if and only if the spherical Fourier transform $(\mu^\natural)^\sim$ is slowly decreasing. Now for any $\lambda\in\a^*_\cc$, the left $K$-invariance of $\varphi_{-\lambda}$ and the Fubini Theorem for distributions implies that
	\begin{align*}
	(\mu^\natural)^\sim(\lambda)&=\int_X \varphi_{-\lambda}(x)\,d\mu^\natural(x)\\
	&=\int_X \varphi_{-\lambda}(x)\,d\mu(x)\\
	&=\int_X\int_B e^{(-i\lambda+\rho)A(x,b)}\,db\,d\mu(x)\\
	&=\int_B\int_X e^{(-i\lambda+\rho)A(x,b)}\,d\mu(x)\,db\\
	&=\int_B \widetilde{\mu}(\lambda,b)\,db.
	\end{align*}
 This finishes the proof.
\end{proof}

\section{Fundamental Solutions of Convolution Operators}
Since the exponential map from $\mathfrak p$ onto $X$ is a diffeomorphism, the topology of $\mathcal D(X)$ is the same as the topology of $\mathcal D(\mathfrak{p})$. Now let $\mathcal D_K(X)$ denote the space of left $K$-invariant functions in $\mathcal D(X)$. Likewise we let $\mathcal D_W(\a)$ be the subspace of $\mathcal D(\a)$ consisting of its $W$-invariant elements. We endow these subspaces with the induced topologies.

By definition, a projection on a topological vector space is just a continuous linear operator $P$ on that space such that $P^2=P$. The averaging operator $f\mapsto f^\natural$ is a projection from $\mathcal D(X)$ onto $\mathcal D_K(X)$.
Likewise, averaging over $W$, we get a projection from $\mathcal D(\a)$ onto $\mathcal D_W(\a)$.

While a closed subspace of an LF space is not  necessarily an LF space, the following lemma shows that the fixed point subspace of a projection of an LF space is one as well.

\begin{lemma}\label{T:lfspace}
Let $\mathcal D$ be an LF space, and let the sequence of Fr\'echet spaces $\{\mathcal D_j\}$ be a sequence of definition of $\mathcal D$. Suppose that $P\colon\mathcal D\to\mathcal D$ is a continuous linear map such that $P^2=P$ and $P(\mathcal D_j)\subset\mathcal D_j$, for all $j$. Then the fixed point subspace $\mathcal D^P$ (with the induced topology) is an LF space, with sequence of definition $\{\mathcal D_j^P\}$.
\end{lemma}
\begin{proof}
Note that  $\mathcal D_j^P=P(\mathcal D_j)$ for all $j$, and that $P(\mathcal D)=\mathcal D^P=\bigcup_j \mathcal D_j^P$.

The inclusion maps $\mathcal D_j^P\hookrightarrow \mathcal D$ are all continuous, so the inductive limit topology on $\mathcal D^P$ is finer than its subspace topology induced from $\mathcal D$.

On the other hand, suppose that $\mathcal U$ is a convex open neighborhood of $0$ in the inductive limit topology of $\mathcal D^P$. We claim that $\mathcal V=P^{-1}(\mathcal U)$ is a neighborhood of $0$ in $\mathcal D$. Since $\mathcal V\cap \mathcal D^P=\mathcal U$, this will prove that the subspace topology on $\mathcal D^P$ is finer than the inductive limit topology.

Now $\mathcal V$ is certainly convex, and since $P(\mathcal D_j)\subset \mathcal D_j$ for each $j$, we have
\begin{align*}
\mathcal V\cap \mathcal D_j&=P^{-1}(\mathcal U)\cap \mathcal D_j\\
&=(P|_{\mathcal D_j})^{-1}(\mathcal U\cap\mathcal D_j^P)
\end{align*}
But $\mathcal U\cap \mathcal D_j^P$ is open (by definition) in $\mathcal D_j^P$ and $P|_{\mathcal D_j}\colon \mathcal D_j\to\mathcal D_j^P$ is continuous. Hence $\mathcal V\cap\mathcal D_j$ is open in $\mathcal D_j$ for each $j$, so $\mathcal V$ is open in $\mathcal D$. This completes the proof.
\end{proof}

\begin{theorem}\label{T:distribution-surjectivity}
Suppose that $\mu\in\mathcal E_K'(X)$. Then the right convolution operator $c_\mu\colon\mathcal D'_K(X)\to\mathcal D'_K(X)$ is surjective if and only if the spherical Fourier transform $\widetilde\mu(\lambda)$ is slowly decreasing in $\a^*_\cc$.
\end{theorem}

\begin{proof} The Paley-Wiener Theorem for the spherical Fourier transform (see \cite{GGA}, Chapter IV) states that the image of $\mathcal D_K(X)$ under the spherical Fourier transform is the space of $W$-invariant holomorphic functions of exponential type in $\a_\cc^*$ which are rapidly decreasing in $\a$. Since this is also precisely the image of $\mathcal D_W(\a)$ under the (Euclidean) Fourier-Laplace transform on $\a$, the projection-slice theorem \eqref{E:abelslice} shows that the Abel transform $\mathcal A\colon\mathcal D_K(X)\to\mathcal D_W(\a)$ is a linear bijection.

By Lemma \ref{T:lfspace}, both $\mathcal D_K(X)$ and $\mathcal D_W(\a)$ are LF spaces. 
Now the Abel transform $\mathcal A\colon \mathcal D_K(X)\to\mathcal D_W(\a)$ is continuous by Lemma 3.5, Ch. I in \cite{GGA}. Hence the Open Mapping Theorem, which holds for LF spaces (\cite{DieudonneSchwartz1949}, Th\'eor\`eme I) implies that $\mathcal A$ is a homeomorphism.

Now the adjoint of the commutative diagram
\begin{equation}\label{E:commdiag2}
\begin{tikzcd}
\mathcal D_K(X) \arrow{r}{c_\mu} \arrow[swap]{d}{\mathcal A} & \mathcal D_K(X) \arrow{d}{\mathcal A} \\
\mathcal D_W(\a) \arrow{r}{c_{\mathcal A\mu}}& \mathcal D_W(\a)
\end{tikzcd} 
\end{equation}
 is the commutative diagram
\begin{equation}\label{E:commdiag3}
\begin{tikzcd}
\mathcal D'_K(X)& \arrow[l,"c_{\widecheck\mu}", swap]\mathcal D'_K(X)\\
\arrow[u,"\mathcal T"]\mathcal D'_W(\a)&\arrow[l,"c_{(\mathcal A\mu)^\vee}",swap]\mathcal D'_W(\a)\arrow[u,"\mathcal T", swap] 
\end{tikzcd}
\end{equation}
where $\mathcal T$ is the adjoint of $\mathcal A$. Since $\mathcal A$ is a homeomorphism, $\mathcal T$ is a bijection.

Suppose that $\widetilde\mu(\lambda)$ is slowly decreasing. Then $(\widecheck\mu)^\sim(\lambda)=\widetilde\mu(-\lambda)$ is slowly decreasing. Now this equals $((\mathcal A\mu)^\vee)^*(\lambda)$, so by \cite{Ehrenpreis1960}, Theorem 2.2, the convolution operator
$$
c_{(\mathcal A\mu)^\vee}\colon\mathcal D'(\a)\to\mathcal D'(\a)
$$
is surjective. Since $\mathcal A\mu$ is $W$-invariant, we can take averages over $W$ to conclude that
$$
c_{(\mathcal A\mu)^\vee}\colon\mathcal D'_W(\a)\to\mathcal D'_W(\a)
$$
is surjective. The diagram \eqref{E:commdiag3} then shows that $c_{\widecheck\mu}\colon\mathcal D'_K(X)\to\mathcal D'_K(X)$ is surjective. Interchanging $\mu$ and $\widecheck\mu$, we conclude that $c_\mu$ is also surjective.

Conversely, suppose that $c_\mu\colon\mathcal D'_K(X)\to\mathcal D'_K(X)$ is surjective. Then the diagram \eqref{E:commdiag3} (with $\mu$ replaced by $\widecheck\mu$) shows that $c_{\mathcal A\mu}\colon\mathcal D'_W(\a)\to\mathcal D'_W(\a)$ is surjective. Lemma \ref{T:W-inv-surj} then implies that $(\mathcal A\mu)^*(\lambda)=\widetilde\mu(\lambda)$ is slowly decreasing.

\end{proof}

In the proof above, we used the fact that the Abel transform $\mathcal A\colon\mathcal D_K(X)\to\mathcal D_W(\a)$ is a homeomorphism. Since $\mathcal D_K(X)\approx\mathcal D_K(\p)$, a simpler homeomorphism is given by the restriction map $f\mapsto f|_{\a}$ from $\mathcal D_K(\p)$ onto $\mathcal D_W(\a)$.

Let $\mathbb D(X)$ denote the algebra of left $G$-invariant differential operators on $X$, and let $\Gamma\colon\mathbb D(X)\to S_W(\a)$ be the Harish-Chandra isomorphism, where $S_W(\a)$ is the algebra of $W$-invariant elements of the symmetric algebra $S(\a)$ (or the algebra of $W$-invariant polynomial functions on $\a^*_\cc$). Each $D\in\mathbb D(X)$ can be thought of as a convolution operator $c_\mu$, where $\mu=D\,\delta_o$. Then for this distribution $\mu$, we have $\widetilde\mu(\lambda)=\Gamma(D)(i\lambda)$, for all $\lambda\in\a^*_\cc$.

\begin{corollary}\label{T:fundamentalsoln2}
Let $D$ be a nonzero element of $\mathbb D(X)$. Then $D$ has a fundamental solution.
\end{corollary}
\begin{proof}
The function $\lambda\mapsto \Gamma(D)(i\lambda)$ is a polynomial function on $\a^*_\cc$, hence slowly decreasing.
\end{proof}
The existence of fundamental solutions of invariant differential operators on symmetric spaces was proved by Helgason in 1964 (\cite{Helgason1963}, \cite{Helgason1964}).

It would be interesting (and natural) to try to extend Theorem \ref{T:distribution-surjectivity} to the surjectivity of $c_\mu$ on all of $\mathcal D'(X)$. There is no proof that we know of at present. In the case when $c_\mu$ is a left-invariant differential operator $D$, Eguchi (\cite{Eguchi1979}, Theorem 10) claims that $D(\mathcal D'(X))=\mathcal D'(X)$, but as Professor Helgason pointed out to us, his proof rests on the unproven claim that the subspace $D^*(\mathcal D(X))$ of the $LF$ space $\mathcal D(X)$ is likewise an $LF$ space.

For a fixed $\mu\in\mathcal E'_K(X)$, let us say that a distribution $S$ on $X$ is a \emph{fundamental solution of $c_\mu$} if $S*\mu=\delta_o$. Theorem \ref{T:distribution-surjectivity} says that if $\widetilde\mu(\lambda)$ is slowly decreasing, then $c_\mu$ has a fundamental solution. The converse is actually also true. To prove it, we need a few preliminary results.

\begin{theorem}\label{T:thmofsupports} (The Theorem of Supports.)
Let $S,\,T\in\mathcal E'(\rr^n)$. If $\text{supp}$ denotes support and $\text{ch}$ denotes the convex hull in $\rr^n$, then
$$
\text{ch}(\text{supp}\,(S*T))=\text{ch}(\text{supp}\,S)+\text{ch}(\text{supp}\,T).
$$
\end{theorem}
This theorem, due to J-L. Lions, is a standard result in distribution theory. See, for example, \cite{Ho1}, Theorem 4.3.3.

\begin{lemma}\label{T:support-result} Fix $\mu\in\mathcal E'_K(X)$ and let $E\subset \mathcal E'_K(X)$. If there is a ball $B_R(o)$ of $X$ containing the supports of all the distributions in $c_\mu(E)$, then there is a ball $B_{R'}(o)$ containing the supports of all the distributions in $E$. 
\end{lemma}
\begin{proof}
If we replace $\mu$ by $\widecheck\mu$ in \eqref{E:commdiagE'}, we obtain the commutative diagram 
\begin{equation}\label{E:commdiag4}
\begin{tikzcd}
\mathcal E'_K(X) \arrow{r}{c_\mu} \arrow[swap]{d}{\mathcal A} & \mathcal E'_K(X) \arrow{d}{\mathcal A} \\
\mathcal E'_W(\a) \arrow{r}{c_{\mathcal A\mu}}& \mathcal E'_W(\a),
\end{tikzcd}
\end{equation}
where the Abel transform $\mathcal A$ is a linear bijection. The hypothesis thus implies that the distributions in the image $\mathcal A(c_\mu(E))=c_{\mathcal A\mu}(\mathcal A(E))$ are all supported in the ball $B_R(0)\subset \a$. By the Theorem of Supports, the distributions in $\mathcal A(E)$ are supported in some ball $B_{R'}(0)$. By the forward Paley-Wiener Theorem and the relation \eqref{E:abelslice}, we see that the spherical Fourier transforms $\widetilde S(\lambda)$, for $S\in E$, are all of exponential type $R'$. The Paley-Wiener Theorem for the spherical Fourier transform on $K$-invariant distributions then implies that the distributions in $E$ are all supported in $B_{R'}(o)$.
\end{proof}


If we endow $\mathcal D'(X)$ with either the weak* or strong topology, then the corresponding subspace topology of $\mathcal D'_K(X)=(\mathcal D'(X))_K$ coincides with its weak* or strong topology as the dual space of $\mathcal D_K(X)$. The same goes for $\mathcal E_K'(X),\;\mathcal E'_W(\a)$, etc. Recall that the bounded sets are the same for the strong and weak* topologies in the dual spaces.

\begin{theorem}\label{T:PW-boundedness}
Let $\mathcal B\subset \mathcal E'(\rr^n)$. Then $\mathcal B$ is bounded in $\mathcal{E'}(\rr^n)$ if and only if there are positive constants $A$ and $C$ and a nonnegative integer $N$ such that
\begin{enumerate}
\item[(i)] The Fourier transform $S^*(\zeta)$ is of exponential type $A$, for every $S\in\mathcal B$, and
\item[(ii)] for all $S\in\mathcal B$ and all $\xi\in\rr^n$, we have
$$
|S^*(\xi)|\leq C\,(1+\|\xi\|)^N.
$$
\end{enumerate}
\end{theorem}

See \cite{EhrenpreisAnnals1956}, Theorem 7.

\begin{theorem}\label{T:fundamentalsoln1}
Fix $\mu\in\mathcal E'_K(X)$. Then $c_\mu$ has a fundamental solution if and only if its spherical Fourier transform $\widetilde\mu(\lambda)$ is slowly decreasing.
\end{theorem}
\begin{proof}
If $\widetilde\mu(\lambda)$ is slowly decreasing, then Theorem \ref{T:distribution-surjectivity} guarantees the existence of $S\in\mathcal D'_K(X)$ such that $S*\mu=\delta_o$.

The converse is the hard part. Suppose that there exists an $S\in\mathcal D'(X)$ such that $S*\mu=\delta_o$. Replacing $S$ by $S^\natural$, we may assume that $S\in\mathcal D'_K(X)$. We wish to prove that $\widetilde\mu$ is slowly decreasing. For this, we first make the following claim:

\textit{The map $\Psi*\mu\mapsto \Psi$ from $c_\mu(\mathcal E_K'(X))$ to  $\mathcal E_K'(X)$ takes bounded sets to bounded sets.}

We now prove the claim. Suppose that $B\subset \mathcal E_K'(X)$ such that $c_\mu(B)$ is bounded. Then there is a ball $B_R(o)$ containing the supports of all the elements of $c_\mu(B)$, so by Lemma \ref{T:support-result}, there is a ball $B_{R'}(o)$ containing the supports of all the elements of $B$.

Now since convolution on $\mathcal E_K'(X)$ is commutative, we have
$$
S*c_\mu(B)=S*\mu*B=\delta_o*B=B*\delta_o=B.
$$
The convolution operator $c_S\colon\mathcal E_K'(X)\to\mathcal D_K'(X)$ is continuous, so $B$ is a bounded subset of $\mathcal D_K'(X)$. Since the elements of $B$ are all supported in the same compact set $\overline B_{R'}(o)$, we conclude that in fact $B$ is bounded in $\mathcal E_K'(X)$, proving the claim.

We will now show that any distribution $\mu\in\mathcal E_K'(X)$ satisfying the claim is slowly decreasing. Since $\mathcal A\colon\mathcal E_K'(X)\allowbreak \to\mathcal E_W'(\a)$ is a homeomorphism, the diagram \eqref{E:commdiag4} shows that the map $T*\mathcal A\mu\to T$ from $\mathcal E'_W(\a)*\mathcal A\mu$ to $\mathcal E'_W(\a)$ takes bounded sets to bounded sets.

Now suppose that $\widetilde\mu$ is not slowly decreasing. All we need is to produce a sequence $\{E_j\}$ contradicting the claim, i.e.
\begin{enumerate}
	\item $\{E_j\}$ is an unbounded set in $\mathcal E'_W(\a)$;
	\item $\{E_j*\mathcal A\mu\}$ is a bounded set in $\mathcal E'_W(\a)$.
\end{enumerate}
For this, we will follow Ehrenpreis' proof in \cite{Ehrenpreis1960}, Theorem 2.2 and adapt it to the $W$-invariant situation.

All of our analysis  shifts at this point to $\a^*$ and $\a^*_\cc$, so for simplicity we will identify these spaces with $\rr^n$ and $\cc^n$, respectively. Moreover, we will not use any special properties of $W$, so we'll just assume that it is a finite subgroup of $\text O(n)$.

Since $\widetilde\mu=(\mathcal A\mu)^*$ is not slowly decreasing, there is a sequence $\{\xi_j\}$ of points in $\a^*=\rr^n$, such that for each $j$ and for all $\xi\in\a^*=\rr^n$ satisfying $\|\xi-\xi_j\|<2j\,\log(2+\|\xi_j\|)$,
\begin{equation}\label{E:mu-estimate}
|\widetilde\mu(\xi)|\leq (2j+\|\xi_j\|)^{-2j}.
\end{equation}
Notice that since $\widetilde\mu$ is $W$-invariant, the same condition holds if we replace $\xi_j$ by $\sigma\cdot\xi_j$, for any $\sigma\in W$. The sequence $\{\xi_j\}$ necessarily satisfies $\|\xi_j\|\to\infty$ as $j\to\infty$.

Now for each $j=1,2,\ldots$, define the function $h_j$ on $\cc$ by
$$
h_j(z)=\left(\frac{\sin(\pi z/j)}{(\pi z/j)}\right)^{2j}.
$$
The functions $h_j$ satisfy the following properties:
\begin{enumerate}
	\item All $h_j$ are entire functions in $\cc$ of exponential type $2\pi$.	
		Actually, we have a uniform estimation for $\{h_j\}$: for all $j$ and all $z\in\cc$,
		\begin{equation}\label{hj:unif-exp}
			|h_j(z)|\leq e^{2\pi|z|}.
		\end{equation}
	\item $h_j(0)=1$.
	\item $0\leq h_j(x)\leq 1$ for all $x\in\rr$.
	\item $h_j(x)\leq \pi^{-2j}$ for $x\in\rr$, $|x|\geq j$.
\end{enumerate}
Properties (1)-(4) are quite straightforward, and
we justify the uniform exponential estimate (1) as follows:

Denote $w=\pi z/j\in\cc$ temporarily. For all $j\geq 1$ and all $z\in\cc^\times$, 
$$ 
	|h_j(z)| = \left|\frac{\sin w}{w}\right|^{2j} =  \left|\frac{\sin w}{w}\right|^{\frac{2\pi|z|}{|w|}} = \left(\left|\frac{\sin w}{w}\right|^\frac{1}{|w|}\right)^{2\pi|z|} := e^{2\pi|z| F(w)} ,	
$$
where 
\begin{align*}
F(w) &= \log \left|\frac{\sin w}{w}\right|^{\frac{1}{|w|}}	\\
&= |w|^{-1}\log \left| 1 - \frac{w^2}{3!} + \frac{w^4}{5!} - \cdots \right|	\\
&\leq |w|^{-1}\log \left( 1 + \frac{|w|^2}{3!} + \frac{|w|^4}{5!} + \cdots \right) \\
&\leq \frac{|w|}{3!} + \frac{|w|^3}{5!} + \cdots \\
&= \sinh (|w|) - |w|.
\end{align*}
When $|w|\leq 1$, since $\sinh(t)-t$ is increasing, then $F(w)\leq \sinh (1)-1<1$ and hence $|h_j(z)|\leq e^{2\pi|z|}$.

On the other hand, when $|w|>1$, using the fact $|\sin w|\leq e^{|\text{Im}\, w|}$ we obtain
\begin{align*}
|h_j(z)| &= \left|\frac{\sin w}{w}\right|^{2j} 	\\
&\leq |w|^{-2j}\left(e^{|\text{Im}\, w|}\right)^{2j}	\\
&\leq e^{2j\, |\text{Im}\,w |}	\\
&= e^{2\pi\, |\text{Im}\,z|}.
\end{align*}
Combining the two cases above, we have shown \eqref{hj:unif-exp}.

Next, let us define the functions $H_j$ on $\cc^n$ by
$$
	H_j(\zeta)=h_j(\zeta^{(1)})\,h_j(\zeta^{(2)})\cdots\,h_j(\zeta^{(n)}),\quad \text{for } \zeta = (\zeta^{(1)},\ldots,\zeta^{(n)})\in\cc^n.
$$
Then $\{H_j\}$ has the following properties:
\begin{enumerate}
	\item[($1'$)] All $H_j$ are entire functions in $\cc^n$ of exponential type $2\pi\sqrt n$.
		To be more precise, for all $j$ and all $\zeta\in\cn$,
		\begin{equation}\label{Hj:unif-exp}
			|H_j(\zeta)|\leq e^{2\pi\sqrt{n}\,\|\zeta\|}.
		\end{equation}
	\item[($2'$)] $H_j(0)=1$.
	\item[($3'$)] $0\leq H_j(\xi)\leq 1$ for all $\xi\in\rr^n$.
	\item[($4'$)] $H_j(\xi)\leq \pi^{-2j}$ if $\xi\in\rr^n$ and at least one coordinate of $\xi$ is $\geq j$.
\end{enumerate}

Properties $(1')$-$(4')$ follow directly from properties (1)-(4) of $\{h_j\}$, just noting that to show \eqref{Hj:unif-exp} we use the elementary inequality
$$
	|\zeta^{(1)}|+|\zeta^{(2)}|+\cdots+|\zeta^{(n)}|\leq \sqrt{n}\,\|\zeta\|.
$$


For each $j$ and each $\sigma\in W$, we next define the function $F_j^\sigma$ on $\cc^n$ by
\begin{equation}
	F_j^\sigma(\zeta)=e^k\,H_k(\sqrt n(\zeta-\sigma\cdot\xi_j)),
\end{equation}
where $k$ is the greatest integer no more than $2j\log (2+\|\xi_j\|)$.
Then the $F_j^\sigma$ satisfy the following properties:
\begin{enumerate}
	\item[($1''$)] All $F_j^\sigma$ are entire functions of exponential type $2n\pi$.
		Precisely, for all $\zeta\in\cc^n$ we have
		\begin{equation}\label{Fj:exp}
			|F_j^\sigma(\zeta)| \leq C_j\, e^{2\pi n\|\zeta\|},
		\end{equation}
		where $C_j = (2+\|\xi_j\|)^{2j} e^{2\pi n\|\xi_j\|}$.
	\item[($2''$)] $F_j^\sigma(\sigma\cdot\xi_j)\geq e^{-1} (2+\|\xi_j\|)^{2j}$.
	\item[($3''$)] $0\leq F_j^\sigma(\xi)\leq (2+\|\xi_j\|)^{2j}$ for all $\xi\in\rr^n$.
	\item[($4''$)] $F_j^\sigma(\xi)\leq 1$ for all $\xi\in\rr^n$ such that $\|\xi-\sigma\cdot \xi_j\|\geq 2j\log(2+\|\xi_j\|)$.
\end{enumerate}

Again, ($1''$)-($4''$) follow from properties ($1'$)-($4'$) respectively. We only explain ($4''$) a bit.

Suppose $\|\xi-\sigma\cdot\xi_j\|\geq 2j\,\log(2+\|\xi_j\|)$  as in ($4''$), then $$\|\sqrt n(\xi-\sigma\cdot \xi_j)\|\geq \sqrt n\, 2j\log (2+\|\xi_j\|)\geq \sqrt{n}\,k$$
which means at least one of the coordinates of $\sqrt n\,(\xi-\sigma\cdot\xi_j)$ has absolute value $\geq k$. Thus, we get ($4''$) from ($4'$) directly:
\begin{align*}
	F_j^\sigma(\xi) &= e^k\,H_k(\sqrt n(\xi-\sigma\cdot\xi_j))	\\
	&\leq e^k \pi^{-2k} \leq 1.
\end{align*}


For each $j$ we now set
\begin{equation}\label{E:Fj-def}
F_j(\zeta) = \frac{1}{|W|} \sum_{\sigma\in W} F_j^\sigma(\zeta),\qquad \zeta\in \cc^n.
\end{equation}
Then for each fixed $j$, $F_j$ is a $W$-invariant entire function on $\cc^n$ of exponential type $2n\pi$, which (by ($3''$)) is bounded (with bound depending on $j$) and nonnegative on $\rr^n$. From ($2''$), we have 
\begin{equation}\label{E:fj}
F_j(\xi_j)\geq e^{-1} (2+\|\xi_j\|)^{2j}.
\end{equation}

For each $j$, consider the unique distribution $E_j\in\mathcal E'_W(\rr^n)$ such that $E_j^*=F_j$. Theorem \ref{T:PW-boundedness} and the inequality \eqref{E:fj} show that the sequence $\{E_j\}$ is not bounded in $\mathcal E'_W(\rr^n)$.

We will prove, on the other hand, that the sequence $\{E_j*\mathcal A\mu\}$ is bounded in $\mathcal E'_W(\rr^n)$. For this, we fix $j$ and consider any  $\xi\in\rr^n$ such that $\|\xi-\sigma\cdot\xi_j\|\geq 2j\,\log (2+\|\xi_j\|)$ for all $\sigma\in W$. Then Property ($4''$) shows that
\begin{equation}\label{E:PW2}
|\widetilde\mu(\xi)\,F_j(\xi)|\leq |\widetilde\mu(\xi)|.
\end{equation}

Next consider any $\xi\in\rr^n$ such that $\|\xi-\sigma\cdot\xi_j\|<2j\,\log (2+\|\xi_j\|)$ for some $\sigma\in W$. Then Property ($3''$), together 
with the inequality \eqref{E:mu-estimate}, with $\xi_j$ replaced by $\sigma\cdot\xi_j$, show that
\begin{equation}\label{E:PW3}
|\widetilde\mu(\xi)\,F_j(\xi)|\leq 1.
\end{equation}

The estimates \eqref{E:PW2} and \eqref{E:PW3} show that for all $\xi\in\rr^n$ and all $j$, we have
\begin{equation}\label{E:PW4}
|\widetilde\mu(\xi)\,F_j(\xi)|\leq |\widetilde\mu(\xi)|+1.
\end{equation}

Now the entire function $\widetilde\mu$ is of some exponential type $A$, so the functions in the sequence $\{\widetilde\mu\,F_j\}$ are all of exponential type $2n\pi+A$.
Moreover, since $\widetilde\mu$ is polynomially increasing in $\rr^n$, the inequality \eqref{E:PW4} implies that there is a constant $C$ and an nonnegative integer $N$ such that
$$
|\widetilde\mu(\xi)\,F_j(\xi)|\leq C(1+\|\xi\|)^N
$$
for all $\xi\in\rr^n$ and all $j$. 

Recall that $E_j^*=F_j$, hence $(E_j*\mathcal A\mu)^* = \widetilde\mu F_j$. 
Theorem \eqref{T:PW-boundedness} therefore implies that the sequence $\{E_j*\mathcal A\mu\}$ is bounded in $\mathcal E'_W(\rr^n)$. 

However, we have already shown $\{E_j\}$ is unbounded in $\mathcal E'_W(\rr^n)$ by \eqref{E:fj}. This contradicts our earlier conclusion that the map $T*\mathcal A\mu\to T$ from $\mathcal E'_W(\a)*\mathcal A\mu$ to $\mathcal E'_W(\a)$ takes bounded sets to bounded sets, and completes the proof of Theorem \ref{T:fundamentalsoln1}.
\end{proof}

We call $\mu\in\mathcal E_K'(X)$ \emph{invertible} if its spherical Fourier transform $\widetilde\mu(\lambda)$ is slowly decreasing. Theorem \ref{T:fundamentalsoln1} then says that $\mu$ is invertible if and only if $c_\mu$ has a fundamental solution. Actually, we can generalize the theorem slightly as follows.

\begin{theorem}
A distribution $\mu\in\mathcal E_K'(X)$ is invertible if and only if there is an invertible distribution $S\in\mathcal E'_K(X)$ and a distribution $T\in \mathcal D'(X)$ such that $T*\mu=S$.
\end{theorem}

\textit{Remark.} Theorem \ref{T:fundamentalsoln1} corresponds of course to the case $S=\delta_o$. Note that we cannot just convolve both sides of $T*\mu=S$ with a fundamental solution to $c_S$, since the resulting left hand side would be a convolution of three distributions, two of which may not have compact support.

\begin{proof}
The ``only if'' part follows immediately from Theorem \ref{T:distribution-surjectivity}.

Conversely, suppose that $S$ and $T$ exist. Replacing $T$ by $T^\natural$, we can assume that $T\in\mathcal D_K'(X)$. 

  We claim that $\mu$ satisfies the claim in the proof of Theorem \ref{T:fundamentalsoln1}; that is to say, we claim that the map $\Psi*\mu\to\Psi$ from $c_\mu(\mathcal E_K'(X))$ to $\mathcal E_K'(X)$ takes bounded sets to bounded sets. The proof of Theorem \ref{T:fundamentalsoln1} will then show that $\mu$ is invertible.

So suppose that $E\subset \mathcal E_K'(X)$ and that $E*\mu$ is bounded in $\mathcal E'_K(X)$.   Note that Lemma \ref{T:support-result} implies that the elements of $E$ all have support inside the same compact subset of $X$.
Since left convolution by $T$ is continuous from $\mathcal E_K'(X)$ to $\mathcal D'_K(X)$, we see that $T*(E*\mu)$ is a bounded subset of $\mathcal D_K'(X)$.

But the convolution of $K$-invariant distributions in $X$ is commutative, so we have
$$
T*E*\mu=E*T*\mu=E*S
$$
Thus $E*S$ is a bounded subset of $\mathcal D'_K(X)$ consisting of distributions which all have support in the same compact set, so it follows that $E*S$ is bounded in $\mathcal E_K'(X)$.

Since $S$ is invertible, Theorem \ref{T:fundamentalsoln1} shows that there is a $\Phi\in\mathcal D'_K(X)$ such that $\Phi*S=\delta_o$. If we apply the exact same argument as above with $\mu$ replaced by $S$ and $T$ replaced by $\Phi$, we conclude that $E*\delta_o=E$ is bounded in $\mathcal E_K'(X)$. Thus $\mu$ satisfies the claim, and it follows that $\mu$ is invertible.
\end{proof}

\bibliographystyle{alpha}
\bibliography{Gonzalez}
%
%

%
%
%
%
%
%
%
\end{document}